\documentclass[absolute]{mymathart}
\usepackage{mymathmacros}

\numberwithin{equation}{section}

 \newtheoremstyle{numberedstyle}% name
   {9pt}%      Space above, empty = `usual value'
   {9pt}%      Space below
   {\normalfont}% Body font
   {}%         Indent amount (empty = no indent, \parindent = para indent)
   {\bfseries}% Thm head font
   {.}%        Punctuation after thm head
   {\newline}% Space after thm head: \newline = linebreak
   {}%         Thm head spec

\newtheorem{thm}{Theorem}[section]%
\newtheorem{lem}[thm]{Lemma}%
\newtheorem{cor}[thm]{Corollary}%
\newtheorem{prop}[thm]{Proposition}%

\newcommand{\edim}{\operatorname{edim}}

\theoremstyle{numberedstyle}
\newtheorem{question}[thm]{Question}%

\newcommand{\B}{\mathcal{B}}

\newcommand{\tef}{transcendental entire function}

\usepackage{hyperref}

\newcommand{\id}{\operatorname{id}}

\title[Hausdorff dimensions of escaping sets]{%
    Hausdorff dimensions of escaping sets \\
    of transcendental entire functions}
\author{Lasse Rempe}
\address{University of Liverpool \\
  Department of Mathematical Sciences \\
  Liverpool L69 7ZL \\
  UK}
\email{l.rempe@liverpool.ac.uk}

\author{Gwyneth M. Stallard}
\address{The Open University \\
   Department of Mathematics and Statistics \\
   Walton Hall\\
   Milton Keynes MK7 6AA\\
   UK}
\email{g.m.stallard@open.ac.uk}

\thanks{Both authors are supported by the European CODY network. The first author is supported by EPSRC fellowship EP/E052851/1. }

\begin{document}
 \begin{abstract}
  Let $f$ and $g$ be transcendental entire functions, each with a bounded set 
   of singular values,
   and suppose that $g\circ\phi = \psi\circ f$, where $\phi,\psi:\C\to\C$
   are affine. We show that the escaping sets of $f$ and $g$ have the
   same Hausdorff dimension.

  Using a result of the second author, we deduce that there exists a family of transcendental
   entire functions for which the escaping set has Hausdorff dimension equal to
   one.
 \end{abstract}

 \maketitle

 \section{Introduction}

Let $f: \C \to \C$ be a transcendental entire function. The Julia set 
  $J(f) \subset \C$ is the set of points where the family
  $(f^n)$ is not equicontinuous (with respect to the spherical metric).
  For an introduction to the dynamics of transcendental entire functions, 
  see \cite{waltermero,waltervorlesung}.

A number of authors have studied the Hausdorff dimension, $\dim J(f)$,  
  of the Julia set (see \cite{gwynethdimensionsurvey} for a survey).  
  Baker \cite{bakerdomainsofnormality} 
  showed that
  $J(f)$ contains nontrivial
  continua for every transcendental entire function $f$, so 
  in particular $\dim J(f)\geq 1$.
  Although it is known~\cite{stallardentirehausdorff4} that, for each 
  $d \in (1,2]$, there are 
  {\tef}s for which the Hausdorff dimension of the Julia set is equal to $d$,
  it is a well-known open
  question whether the Julia set of a {\tef} can have Hausdorff dimension
  equal to $1$. 

 More is known when we restrict the class of functions under consideration.
  The \emph{Eremenko-Lyubich class} is defined by
  \[
  \B = \{ f : f \text{ is a {\tef} for which } \sing(f^{-1}) \text{ is bounded} \},
  \]
  where $\sing(f^{-1})$ consists of the critical and asymptotic values of $f$.  The second author proved in
  \cite{stallardentirehausdorff2} that
  \begin{equation}\label{dimB}
  \dim J(f) >1 \quad \text{ for $f\in\B$}.
  \end{equation}

 This proof in fact constructs a subset $A$ of the 
 {\it escaping set}
 \[ I(f) := \{z\in\C: f^n(z)\to\infty \text{ as } n \to \infty\} \]
  whose closure $\cl{A}$ has Hausdorff dimension greater than one, and uses
  the fact \cite{alexmisha} that 
  \begin{equation} \label{eqn:classB}
    I(f)\subset J(f)\quad \text{ for $f\in\B$.}\end{equation}

  It was shown more recently
  \cite{ripponstallardfinitepoles} 
  that the set $I(f)\cap J(f)$ contains nontrivial continua for all 
  transcendental entire functions, and hence
  \begin{equation}\label{continua}
  \dim (I(f) \cap J(f)) \geq 1. 
  \end{equation}
  In view of these results, it seems natural
  to ask whether the escaping set of a transcendental entire function
   must have Hausdorff dimension greater than one. We show that 
   this is not the case:

 \begin{thm}[Escaping sets of dimension one] \label{thm:dim1}
  There exists a function $f\in\B$ such that $\dim I(f)=1$. 
 \end{thm}
 
 Combined with previous results, this gives the following complete description
  of the possible Hausdorff dimensions of escaping sets.
 \begin{cor}[Dimensions of escaping sets]  \label{cor:escapingdim}
  If $f$ is a transcendental entire function, then
   $\dim I(f)\in [1,2]$. Conversely, for every $d\in[1,2]$ there exists
  a function $f\in\B$ with $\dim I(f)=d$. If $d>1$, this function can be 
  chosen such that $\dim J(f)=d$. 
 \end{cor}

 To prove Theorem~\ref{thm:dim1} we consider a 
  function that was studied in \cite{stallardentirehausdorff1}. 
  Let $L$ be the boundary of the region
   \[
     G = \{z: \re (z) > 0, - \pi < \im (z) < \pi \},
   \]
 parametrized in clockwise direction. 
  Then
   \begin{equation} \label{eqn:F0}
    F_0:\C\setminus{\cl{G}}\to \C; \quad z\mapsto 
      \frac{1}{2 \pi i}  \int_{L} \frac{\exp(e^t)}{t-z} dt 
   \end{equation}
  can be continued analytically to a transcendental entire function
  $F_0:\C\to\C$. From the properties of $F_0$ given in
  \cite{stallardentirehausdorff1}, it can easily be seen that
  $F_0\in\B$. 

 Consider the family 
 \[
  F_{\kappa}(z) := F_0(z)+\kappa, \quad \kappa\in\C.
 \]
 It is shown in
  \cite{stallardentirehausdorff1} that
  \begin{equation}\label{dim1}
  \lim_{\genfrac{}{}{0pt}{}{\kappa\to-\infty,}{\kappa\in\R}}
     \dim J(F_{\kappa})\ = 1, \end{equation}
  while,  by~\eqref{dimB}, $\dim J(F_{\kappa})>1$ for all $\kappa\in\C$.
  This implies that $\dim J(F_{\kappa})$ is a non-constant
  function of $\kappa$.

 In contrast, we show that the Hausdorff dimension of the escaping set cannot
  change in a family defined in this manner.  
  More precisely, we
  say that two transcendental entire functions $f$ and $g$ are \emph{affinely equivalent}
  if there are
  affine functions $\phi,\psi:\C\to\C$ such that
    \begin{equation} \label{eqn:equivalence} \psi\circ f = g\circ\phi.
    \end{equation}
 Any two functions $F_{\kappa_1}$ and $F_{\kappa_2}$, 
  $\kappa_1, \kappa_2 \in \C$,
  are clearly affinely equivalent. (Another well-known family consisting of 
  transcendental entire functions that are affinely equivalent to each other 
  is the family $z\mapsto \exp(z)+\kappa$ of exponential maps. Note that, for 
  this family, the Julia sets -- and escaping sets --
  all have Hausdorff dimension two by a result of 
  McMullen
  \cite{hausdorffmcmullen}.)

 \begin{thm}[Escaping dimensions and affine equivalence] 
 \label{thm:mainrigidity}
   Suppose that $f,g\in\B$ are affinely equivalent. Then
    $\dim I(f) = \dim I(g)$.
 \end{thm}

 The proof uses recent results of 
  the first author \cite{boettcher} on the
  rigidity of the dynamics near infinity for a function
  $f\in\B$. Using~\eqref{dimB}, \eqref{eqn:classB} and
  \eqref{dim1}, we obtain the following consequence of 
  Theorem~\ref{thm:mainrigidity},
  which implies Theorem \ref{thm:dim1}.

 \begin{cor}[Escaping dimension of $F_{\kappa}$] \label{cor:Fkappa}
  Let $F_0$ be the function defined above, and let $g\in\B$ be 
    affinely equivalent to
   $F_0$. Then $\dim I(g)=1<\dim J(g)$. \qedoutsideproof
 \end{cor}
\begin{remark}[Remarks]\mbox{}
 \begin{enumerate}
  \item
   As far as we know, this provides the first example of an analytic
   family
   of entire functions for which the dimension of the escaping set is
   always strictly smaller than that of the Julia set.
\item
 We note that, 
  if $f$ has finite order, then $I(f)$ has Hausdorff dimension two
  \cite{baranskihausdorff,hendrikthesis}, while for functions
  of ``small'' infinite order, a lower bound on the Hausdorff dimension
  of $I(f)$ is proved in \cite{bergweilerkarpinskastallard}. It follows
  from these results that 
   \[ \limsup_{r\to\infty} \frac{\log\log\log \max_{|z|=r}|f(z)|}{\log \log r}=\infty \]
  for any function with $\dim I(f)=1$. 
  Note that for the maps $F_{\kappa}$ we have
   \[  \log\log\log \max_{|z|=r}|F_{\kappa}(z)| \approx \log r. \] 
\item
 For transcendental \emph{meromorphic} functions,
  the Julia set may have any Hausdorff dimension $d\in (0,2]$ 
  \cite{stallardmeromorphichausdorff2}. 
  On the other hand, \eqref{dimB} was generalized to functions having a 
   \emph{logarithmic singularity} over $\infty$ in \cite{walterphilgwyneth},
   and 
   extended in \cite{baranskikarpinskazdunik} to show that in fact the 
   \emph{hyperbolic dimension} of such a function
   is strictly larger than one. As far as we know, it is an open
   question whether the escaping set of a transcendental meromorphic function
   can have Hausdorff dimension zero --- this cannot occur for the
   Julia set \cite{stallardmeromorphichausdorff}. 
  Kotus and Urba{\'n}ski \cite{kotusurbanskiradialescapingmeromorphic} 
   have shown that there exist meromorphic functions
   with $\dim J(f) > \dim I(f)$; compare also \cite[Theorem 1.2]{bergweilerkotus}.
\end{enumerate}
\end{remark}

 To conclude our paper, we consider the notion of the
  {\it eventual dimension} of a transcendental entire function $f$, defined by
  \[
     \edim(f) := \inf_{R>0} \dim J_R(f),
  \]
  where, for each $R>0$,
  \begin{equation}\label{J_R(f)}
  J_R(f) := \{z\in J(f): |f^n(z)|\geq R\text{ for all $n\geq 1$}\}.
   \end{equation}
 \begin{remark}
  This eventual dimension is also implicitly used by 
   Bergweiler and Kotus \cite{bergweilerkotus} to obtain an
   upper bound for the Hausdorff dimension of the escaping set of
   certain meromorphic functions. 
 \end{remark}
   
   We show that, for functions in the Eremenko-Lyubich class,
    the eventual dimension is an upper bound for
    $\dim I(f)$ and is also preserved under affine equivalence. This 
    enables us to prove the following result.

  \begin{thm}[Eventual dimensions of entire functions]\label{thm:edim}
  For each $d \in [1,2]$, there exists a function $f\in\B$ 
   such that $\edim(g) = \dim I(g) = d$ for all functions $g$ affinely 
   equivalent to $f$.
  \end{thm}

 \subsection*{Open questions}
 
 As far as we are aware, for all examples where
  $\dim I(f)$ and $\edim(f)$ are known, these numbers coincide.
  (This is also the case for the meromorphic functions considered in
   \cite[Theorem 1.2]{bergweilerkotus}.)
 
 \begin{question}
  Does there exist a function $f\in\B$ such that $\edim(f)\neq\dim I(f)$?
 \end{question}

  Following \cite{alexmisha}, two entire functions $f$ and $g$
   are called
   \emph{quasiconformally equivalent} if there are quasiconformal
   functions $\psi,\phi:\C\to\C$ such that (\ref{eqn:equivalence}) holds.
   Quasiconformal equivalence classes can be considered as natural
   parameter spaces. It was shown in \cite{boettcher} that, for $f\in\B$,
   the dynamical behaviour near infinity is the same for any function
   quasiconformally equivalent to $f$.

 \begin{question}
   Suppose that $f,g\in\B$ are quasiconformally but not affinely equivalent.
    Is $\dim I(f)=\dim I(g)$?
 \end{question}
 
 \begin{remark}
  It follows from \cite{boettcher} that the answer is ``yes'' if
   $\dim I(f)=2$.
 \end{remark}

 \subsection*{Structure of the article}
  In Section \ref{sec:boettcher}, 
   we use the results of \cite{boettcher} to deduce a fact
   that will play an important role in our proof of Theorem
   \ref{thm:mainrigidity}, which is itself proved in 
   Section \ref{sec:mainproof}.
   We give the proof of Corollary \ref{cor:escapingdim} in 
   Section \ref{sec:escapingdim}, and treat eventual dimension and
   Theorem \ref{thm:edim} in Section \ref{sec:edim}.

 \section{Conjugacies on $J_R(f)$} \label{sec:boettcher}
  We will require the fact (Corollary \ref{cor:boettcher} below) that, if 
   two 
   functions $f,g\in\B$ are 
   affinely equivalent then, for sufficiently large $R$,
   they are also quasiconformally conjugate on
   the set $J_R(f)$, and the dilatation of the conjugacy tends to 
   one as $R\to\infty$.
  This follows from the ideas of \cite{boettcher}, but is
   not explicitly stated there. We shall provide a proof for completeness, 
   using the following result, which is a special case of
   \cite[Proposition 3.6]{boettcher}.   

  \begin{prop}[Existence of conjugacies] \label{prop:boettcher1}
    Let $f\in\B$. Let 
    $\phi_{\lambda}:\C\to\C$, $\lambda\in\C$,
    be a family of nonconstant affine maps
    that depend analytically on $\lambda$. Also suppose that
    $\phi_0=\id$. We define 
    $f_{\lambda} := f\circ\phi_{\lambda}$. 

  Let $N$ be a compact subset of $\C$ with $0\in N$. 
   Then there exists a constant $R>0$ 
   such that, for every $\lambda\in N$, there is an injective
   function
   $\theta=\theta^{\lambda}:J_R(f) \to J(f_{\lambda})$ with the following
   properties:
   \begin{enumerate}
    \item $\theta^{0} = \id$,
    \item 
      $\theta^{\lambda}\circ f = f_{\lambda} \circ \theta^{\lambda}$ and
    \item  
     for fixed $z\in J_R(f)$, the function $\lambda\mapsto\theta^{\lambda}(z)$ 
     is analytic in $\lambda$ (on the interior of $N$). \qedoutsideproof
   \end{enumerate}
  \end{prop}
 \begin{remark}
  The conclusion of the theorem says that 
   the injections $\theta^{\lambda}$ form a \emph{holomorphic motion}
   of $J_R(f)$ on the interior of $N$ 
   in the sense of Ma\~n\'e, Sad and Sullivan.
   (Compare \cite[Section 5.2]{hubbardteichmueller}). 
 \end{remark}
  
  We now prove the main result of this section.
  
  \begin{cor}[Conjugacy near infinity] 
  \label{cor:boettcher}
   Suppose that $f,g\in\B$ are affinely equivalent, and let
    $K>1$.
   Then there exist $R>0$
    and a $K$-quasiconformal map $\theta:\C\to\C$ such that
     \[ \theta(f(z)) = g(\theta(z)) \]
    for all $z\in J_R(f)$.
  \end{cor}
 \begin{proof}
   By conjugating $g$ with a M\"obius transformation, we may assume
    for simplicity that $g = f\circ M$, where
    $M(z)=e^{A}z + B$ for suitable $A,B\in\C$.
   Define affine functions 
     $\phi_{\lambda}(z) := e^{\lambda A}z + \lambda B$ 
    and consider the family 
     \[ f_{\lambda}(z) := f(e^{\lambda A}z + \lambda B). \]
   Then $f_0=f$ and $f_1=g$.

  Set $D:=(K+1)/(K-1)$ and $N:= \{\lambda:|\lambda|\leq D\}$.
   Let $R>0$ and $\theta^{\lambda}:J_R(f)\to J(f_{\lambda})$,
   $\lambda\in N$, 
   be as in
   Proposition~\ref{prop:boettcher1}.

  It is well-known that a holomorphic motion such as
   $\theta^{\lambda}$ is quasiconformal as a function of $z$,
   and there is a bound on the dilatation of $\theta^{\lambda}$
   in terms of the parameter $\lambda$. 
  More precisely: if $\lambda\in \operatorname{int}(N)$, 
   then 
   $\theta^{\lambda}$ extends to a $K_{\lambda}$-quasiconformal
   map $\theta^{\lambda}:\C\to\C$ by \cite[Theorem 1]{bersroyden}, where
     \[ K_{\lambda} \leq \frac{D+|\lambda|}{D-|\lambda|}. \]
   In particular, by the definition of $D$,
   the map $\theta_1$ extends to a
   $K$-quasiconformal map $\theta_1:\C\to\C$, as desired.
 \end{proof}

 \section{Proof of Theorem \ref{thm:mainrigidity}} \label{sec:mainproof}

Let $f$ be a transcendental entire function. For $R\geq 0$, we defined
 the set $J_R(f)$ in the introduction; let us also set 
\[I_R(f) := \{z \in I(f): |f^n(z)|\geq R\text{ for all $n\geq 1$}\}. \]

 \begin{lem}
 %[Dimension of escaping sets] 
 \label{lem:escapingdimension}
  Let $f$ be a transcendental entire function and $R>0$. Then
   $\dim I(f) = \dim I_R(f)$ and
   $\dim I(f)\cap J(f) = \dim I_R(f)\cap J(f)$. 
 \end{lem}
 \begin{proof}
  Every point in $I(f)$ eventually maps to a point in $I_R(f)$, so we have
    \[ I(f) = \bigcup_{j\geq 0} f^{-j}(I_R(f)), \]
   and analogously for $I(f)\cap J(f)$.
   The conclusion follows from 
    the preservation of Hausdorff dimension under holomorphic functions and
    the countable stability of Hausdorff dimension. 
 \end{proof}

 \begin{proof}[Proof of Theorem \ref{thm:mainrigidity}]
  Suppose that $f,g \in \B$ are affinely equivalent. 
  We want to show that $\dim I(f) = \dim I(g)$. 
  By symmetry, it is sufficient to show that $\dim I(g)\geq \dim I(f)$. 
  We shall use the following 
  result of 
  Gehring and V\"ais\"al\"a \cite[Theorem 8]{gehringvaisala}: 
  if $\theta$ is a $K$-quasi-conformal map and $A\subset\C$, then
     \begin{equation} 
      \dim(\theta(A)) \geq \dim\theta(A)/K. \label{eqn:distortion}
     \end{equation}
  (The optimal bounds on the distortion of Hausdorff dimension under
   quasiconformal mappings in dimension $2$ are given by Astala's distortion
   theorem 
   \cite{astaladistortion}.) 

  So, let $K>1$. 
   It follows from Corollary~\ref{cor:boettcher} that there exist $R>0$ and a 
   $K$-quasiconformal map $\theta:\C\to\C$ such that
      $\theta(f(z)) = g(\theta(z))$
    for all $z\in J_{R}(f)$.  
    
    For
   $z\in I_{R}(f)\subset J_R(f)$, we have
     \[ g^n(\theta(z)) = \theta(f^n(z)) \to \infty, \]
    and hence, by (\ref{eqn:distortion}) and
    Lemma~\ref{lem:escapingdimension},
   \[ \dim I(g)\geq \dim \theta(I_R(f)) \geq
        \dim(I_R(f))/K. \]
   Since $K$ was chosen arbitrarily close to $1$, this completes the proof.
 \end{proof}

 \section{Proof of Corollary~\ref{cor:escapingdim}} \label{sec:escapingdim}
 The first statement in Corollary~\ref{cor:escapingdim} follows from
  \eqref{continua}. Recall that the second claim is that, 
  for each $d \in [1,2]$, 
  there exists a function $f\in\B$
  for which the Hausdorff dimension of the escaping set is equal to $d$,
  and that furthermore if $d>1$, the function can be chosen such that
  $\dim J(f)=d$.

 In the case $d=1$, this follows from Corollary 
  \ref{cor:Fkappa}, while for $d=2$ we can use $f(z)=\exp(z)$ and use 
  McMullen's result \cite{hausdorffmcmullen}
  that $\dim I(f)=2$. (In fact, by
   \cite{baranskihausdorff,hendrikthesis}, 
    any entire function $f\in\B$ of finite order has the desired properties
    for $d=2$.)

 The remaining cases are covered by the following statement, which is
  a consequence of the results of \cite{stallardentirehausdorff4}
  and \cite{bergweilerkarpinskastallard}.

  \begin{prop}[Escaping sets of dimension between one and two]
   For each $d\in (1,2)$, there exists a family of functions
    $f_{d,K}$, $K\in\R$, in the class $\B$ such that
      \[ \dim J(f_{d,K})=  \dim I(f_{d,K}) = d. \]
  \end{prop}
 \begin{proof}
  Set $p=(2-d)/(d-1)$, and define the family by 
 \[
  f_{d,K}(z) := f_p(z) - K \quad \mbox{ for } K\in\R, 
 \]
 where
    \[ f_p(z) := \frac{1}{2 \pi i}  \int_{L_p} \frac{\exp(e^{(\log t)^{1+p})}}{t-z} dt,\]
with $L_p$ being the boundary of the region
\[
G_p = \{z = z + iy: |y| \leq \pi x/[(1+p)(\log (x))^p], \; x \geq 3 \},
\]
described in a clockwise direction, for $z \in \C \setminus \cl{G_p}$. 
 As with the function $F_0$ defined in the introduction,
 $f_p$ can be defined by analytic continuation for $z \in \cl{G_p}$.

 It was shown in \cite{stallardentirehausdorff4} that
  $f_{d,K}\in\B$ and that 
   \[ \dim J(f_{d,K}) = 1 + \frac{1}{1+p} = d \]
  for sufficiently large $K$. By~\eqref{eqn:classB}, we have 
   (again, for large $K$), 
 \begin{equation}\label{bound}
  \dim I(f_{d,K}) \leq \dim J(f_{d,K}) = d.
 \end{equation}

  On the other hand, the following result  was proved 
   in~\cite{bergweilerkarpinskastallard}. Suppose that 
   $f \in \B$, $q\geq 1$, and that, for each $\eps>0$, there exists 
   $r_\eps>0$ such that
 \begin{equation}\label{a}
   |f(z)|\leq \exp\left( \exp\left(\left(\log|z|\right)^{q+\eps}\right)\right)
   \quad \text{for } |z|\geq r_\eps.
  \end{equation}
  Then
    $\dim I(f)\geq 1+\frac{1}{q}$.

 The functions $f_{d,K}$ satisfy the above assumptions for
  $q=1+p = 1/(d-1)$, and hence 
 \[
  \dim I(f_{d,K}) \geq 1 + \frac1q = d 
 \]
  for all $K$.
 Together with \eqref{bound}, this implies that, for large $K$,
 \[
  \dim I(f_{d,K}) = \dim J(f_{d,K}) = d,
 \]
 as claimed.
\end{proof}

 \section{Eventual dimension} \label{sec:edim}
  Recall that the 
  {\it eventual dimension} of an entire
   function $f$ was defined in the introduction as
      \[ \edim(f) = \inf_{R>0} \dim J_R(f). \]
 We begin with the following two results.

   \begin{lem} \label{lem:edim}
    Let $f$ be a transcendental entire function.
     Then $\dim(I(f)\cap J(f))\leq \edim(f) \leq \dim J(f)$.
   \end{lem}
   \begin{proof}
    The second inequality holds by definition. On the other hand, by 
     Lemma \ref{lem:escapingdimension}, we have 
   \[ \dim(I(f)\cap J(f)) = \dim( I_R(f) \cap J(f) ) = 
      \dim (J_R(f)\cap I(f)) \]
     for all $R\geq 0$.
     Hence $\dim I(f)\cap J(f)\leq \dim J_R(f)$ for all
     $R\geq 0$ and thus $\dim I(f)\cap J(f)\leq \edim(f)$, as claimed.
   \end{proof}

  \begin{thm} \label{thm:edim2}
   Suppose that $f,g\in \B$ are affinely equivalent. Then $\edim(f)=\edim(g)$.
  \end{thm}
  \begin{proof}
The proof is very similar to the proof of Theorem~\ref{thm:mainrigidity}. Again, we need only show that
    $\edim(g)\geq \edim(f)$. 
    
    Let $K>1$. 
    It follows from Corollary~\ref{cor:boettcher} that there exist $R>0$ and a 
    $K$-quasi\-con\-formal map $\theta:\C\to\C$ such that
     \begin{equation}\label{theta} \theta(f(z)) = g(\theta(z)) \end{equation}
    for all $z\in J_R(f)$.

    Now let $S>0$ and choose $R'\geq R$ sufficiently large that
     $|\theta(z)|\geq S$ whenever $|z|\geq R'$. Then it follows from 
     \eqref{theta}  that
      $\theta(J_{R'}(f))\subset J_S(g)$. Therefore, by~\eqref{eqn:distortion}
      and Lemma~\ref{lem:escapingdimension}, 
    \[ \dim J_S(g) \geq \dim \theta (J_{R'}(f)) \geq
       \dim J_{R'}(f)/K \geq \edim(f)/K. \]
   Since $K$ can be chosen arbitrarily close to $1$,
    we see that $\dim J_S(g) \geq \edim(f)$. 
    Finally, since $S$ was arbitrary, it follows
    that $\edim(g)\geq \edim(f)$, as required.
   \end{proof}

We end this section by proving Theorem~\ref{thm:edim}. 

\begin{proof}[Proof of Theorem \ref{thm:edim}]
 Let $d\in [1,2]$.
  By Theorems \ref{thm:mainrigidity} and  \ref{thm:edim2}, we only
  need to show that there is a function $f\in\B$ with
  $\dim I(f) = \edim(f) = d$. If $d>1$, then this follows from
  Corollary \ref{cor:escapingdim} together 
  with (\ref{eqn:classB}) and Lemma~\ref{lem:edim}. 

 For $d=1$, we again consider the functions $F_{\kappa}$, $\kappa \in \C$, 
  defined in the Introduction. It follows from \eqref{continua}, \eqref{dim1}, Lemma~\ref{lem:edim} and Theorem~\ref{thm:edim2} that
\[
 \edim(F_{\kappa}) = 1 \quad \text{ for all $\kappa \in \C$}.
\]
Also, by Corollary \ref{cor:Fkappa}, 
\[
 \dim I(F_{\kappa}) = 1 \quad \text{ for all $\kappa \in \C$}.
\]
This completes the proof.
\end{proof}

\bibliographystyle{hamsplain}
\bibliography{/Latex/Biblio/biblio}

\end{document}